\documentclass[]{amsart}
\usepackage{amsthm,amssymb,amsfonts,setspace,verbatim,graphicx,mathtools,mathrsfs,commath,enumitem}

\usepackage[hidelinks]{hyperref}
\newtheorem{ut}{Theorem}
\newtheorem{ul}[ut]{Lemma}
\newtheorem{up}[ut]{Proposition}

\newtheorem{uc}[ut]{Corollary}

\theoremstyle{definition}

\newtheorem{ue}{Example}
\usepackage{hyperref}

\newtheorem{uq}{Question}

\theoremstyle{remark}

\usepackage[margin=1.5in]{geometry}

\newcommand{\interior}[1]{%
  {\kern0pt#1}^{\mathrm{o}}%
}

\begin{document}

\title{Meager composants of tree-like continua}

\author[D.S. Lipham]{David S. Lipham}
\address{Department of Mathematics, Auburn University at Montgomery, United States of America}
\email{dsl0003@auburn.edu}

\begin{abstract}A subset $M$ of a continuum $X$ is called a \textit{meager composant} if $M$ is maximal with respect to the property that every two of its points are contained in a nowhere dense subcontinuum of $X$. 
Motivated by questions of Bellamy, Mouron and Ordo\~{n}ez, we show  that no tree-like continuum has a proper open meager composant, and that every tree-like continuum has either $1$ or $2^{\aleph_0}$ meager composants. We also prove a decomposition theorem: If  $X$ is tree-like and every indecomposable subcontinuum of $X$ is nowhere dense, then the partition of $X$ into meager composants is upper semi-continuous and the space of meager composants is a dendrite.  
\end{abstract}

\maketitle

\section{Introduction}

  A \textbf{continuum} is a compact connected metric space.  Given a continuum $X$ and a point $x\in X$, then  $$M_x=\bigcup\big\{L\subset X:L\text{ is a nowhere dense subcontinuum of }X\text{ and }x\in L\big\}$$ is called the \textbf{meager composant of $x$} in $X$.   More generally,   $M\subset X$ is  a \textbf{meager composant} of $X$ if there exists $x\in X$ such that $M=M_x$.  Note that $$\mathcal M_X=\big\{M_x:x\in X\big\}$$ partitions $X$ into pairwise disjoint sets. This partition is topological in the sense that homeomorphisms respect its members. 
  
  Meager composants of continua were introduced by David Bellamy in \cite{1} as ``tendril classes''. They were subsequently investigated by Chris Mouron and Norberto Ordo\~{n}ez  in \cite{mou}, and by the current author in \cite{sing}.     Bellamy asked if  there exists a continuum with a proper open  meager composant  (i.e.\ a meager composant  which is open in $X$ but not equal to $X$)  \cite[Problem 25]{1}. In this article we will show that there is no tree-like example.  
   
   Mouron and Ordo\~{n}ez proved that $\mathcal M_X$ is  upper semi-continuous if $X$ is locally connected, hereditarily arcwise connected, or irreducible and hereditarily decomposable \cite[Corollary 8.2]{mou}. They asked to identify other classes of continua for which  $\mathcal M_X$ is  upper semi-continuous \cite[Problem 8.1]{mou}. Here we will add the class of tree-like continua whose indecomposable subcontinua are nowhere dense. We will also  prove that each  tree-like continuum has either $1$ or $2^{\aleph_0}$ meager composants.  In general it is not evident that the  cardinality of $\mathcal M_X$ must be  $1$ or $2^{\aleph_0}$  \cite[Problem 8.5]{mou}. There may even be   a continuum  with exactly $2$ meager composants, one of which is a singleton.
    \begin{uq}[{\cite[Problem 8.8]{mou}}]Is there a continuum $X=O\cup \{p\}$ where $O$ and $\{p
\}$ are meager composants of $X$?\end{uq}

\subsection*{Outline of the paper}The only property of tree-like continua that we will use  is hereditary unicoherence (defined in Section 2).  As such, all theorems will be stated and proved for hereditarily unicoherent continua. The  results are organized as follows.

Suppose that  $X$ is a hereditarily unicoherent continuum.
\begin{itemize}\renewcommand{\labelitemi}{\scalebox{.5}{{$\blacksquare$}}}
\item In Section 4 we will prove that $X$ has no proper open meager composant.
\item In Section 5 we will show that every meager composant of $X$ is closed if and only if every indecomposable subcontinuum of $X$ is nowhere dense. 
\item In Section 6 we will show that if every meager composant of $X$ is closed, then $\mathcal M_X$ is upper semi-continuous and the space $\mathcal M_X$ is a dendrite.
\item In Section 7 we will apply  results from Sections 5 and 6 to show $|\mathcal M_X|\in \{1,2^{\aleph_0}\}$. 
\end{itemize}
We suspect that most of these statements are also true in the context of homogeneity. In the Appendix at
the end of the paper, we will begin investigating meager composants of homogeneous continua, and pose
some questions for further research


\section{Definitions and basic notions}

 
 A continuum $X$ is \textbf{tree-like} if for every $\varepsilon>0$ there is an acyclic graph $T$ (a \textbf{tree}) and a mapping $f:X\to T$ such that $\mathrm{diam}(f^{-1}\{t\})<\varepsilon$ for each $t\in T$.   A \textbf{dendrite} is a locally connected continuum which contains no simple closed curve. Dendrites are tree-like \cite[Theorem 10.32]{nad}.
  
 A continuum $X$ is  \textbf{hereditarily unicoherent} if for every two subcontinua $A$ and $B$ the intersection $A\cap B$ is connected.  Tree-like continua are  hereditarily unicoherent \cite[p. 232]{nad}.

The \textbf{composant} of a point $x\in X$ is  defined to be the union of all proper subcontinua of $X$ that contain $x$. The \textbf{meager composant} of $x\in X$ is  the union of all nowhere dense subcontinua of $X$ that contain $x$.

A continuum $X$ is  \textbf{decomposable} if $X$ can be written as the union of two of its proper subcontinua; otherwise $X$ is  \textbf{indecomposable}. Equivalently, $X$ is indecomposable if every proper subcontinuum of $X$ is nowhere dense \cite[Exercise 6.19]{nad}.  In an indecomposable continuum, composants and meager composants  are the same and are not closed \cite[Proposition 11.14]{nad}. 

The decomposition $\mathcal M_X$ is \textbf{upper semi-continuous} if  for every closed $A\subset X$  the union $\bigcup\{M_x:x\in A\}$ is closed in $X$ \cite[Chapter III]{nad}. If $\mathcal M_X$ is upper semi-continuous, then it is a continuum in the quotient topology \cite[Theorem 3.10]{nad}.

\section{Ample propositions}

We begin by proving two very useful propositions which involve the notion of an ample subcontinuum.  A subcontinuum $A$ of a continuum $X$ is \textbf{ample} if for every neighborhood $U$ of $A$, there is a continuum $K$ such that $A\subset \interior{K}\subset K\subset U$ \cite[Definition 2]{pra}.

\begin{up}\label{a}Let $X$ be a continuum and let $K$ be a subcontinuum  $X$. If $K$ is not ample then there is a nowhere dense subcontinuum of $X$ meeting $K$ and $X\setminus K$.\end{up}

\begin{proof} Suppose that $K$ is not ample. Then there is a compact neighborhood $U$ of $K$ such that if $B$ is the connected component of $K$ in $U$, then $K\not\subset \interior{B}$.  
Fix $x\in K\setminus \interior{B}$, and let $x_0\in U\setminus B$ such that $d(x_0,x)<1$. By compactness of $U$  there  is a relatively clopen subset $A_0 $ of $U$ such that $B\subset A_0$ and $x_0\notin A_0$ (cf.\ \cite[Theorem 6.1.23]{eng}). Likewise, assuming  $x_0,\ldots,x_{n-1}$ and clopen sets $A_0\supset \ldots\supset A_{n-1}$ have been defined,  there exists $x_{n}\in A_{n-1}\setminus B$ and a clopen $A_n\subset A_{n-1}$ such that   $d(x_n,x)<1/n$ and  $B\subset A_n\subset X\setminus\{x_n\}$. For each $n\geq 0$ let $C_n$ be the connected component of $x_n$ in $U$. Then $\{C_n:n\geq 0\}$ is a discrete collection of continua in $X$, and every $C_n$ meets $\partial U$ by the boundary bumping principle \cite[Theorem 5.4]{nad}. We conclude that $\textstyle H=\overline {\bigcup _{n=0} ^\infty C_n}\setminus \bigcup _{n=0} ^\infty C_n$ is a nowhere dense subcontinuum\footnote{To see that the set $H$  is connected, it suffices to let $V$ and $W$ be open subsets of $X$ such that $V\cap H\neq\varnothing$, $W\cap H\neq\varnothing$, and $H\subset V\cup W$,  and show  $V\cap W\neq\varnothing$. To that end, assume $x\in V$. Then there is an integer $N_1$ such that $x_n\in V$ for all $n\geq N_1$. By compactness of $X$ there exists $N_2$ such that $C_n\subset V\cup W$ for all $n\geq N_2$. Let $n\geq N_1+N_2$ such that $C_n\cap W\neq\varnothing$. Then $C_n\subset V\cup W$, $C_n\cap V\neq\varnothing$, and $C_n\cap W\neq\varnothing$. Since $C_n$ is connected we have $V\cap W\neq\varnothing$.} of $X$ which contains $x\in K$ and meets  $\partial U\subset X\setminus K$.  \end{proof}

\begin{up}\label{b}Let $M_x$ be a meager composant of a hereditarily unicoherent continuum $X$.  If $K\subset X$ is a continuum and $\overline{M_x\cap K}=K$, then every ample subcontinuum of $K$ intersects $M_x$.\end{up}

\begin{proof} Let $K\subset X$ be a continuum such that $M_x\cap K$ is dense in $K$. Let $A$ be an ample subcontinuum of $K$. We may assume $A$ is nowhere dense, otherwise it trivially intersects $M_x$. Let $K_0\supset K_1\supset\ldots$  be a decreasing sequence of continua in $K$ such that $A\subset K_n^{\mathrm{o}}$ (the interior in $K$) and $\bigcap_{n=0}^\infty K_n=A$.  For each $n\geq 0$ choose $x_n\in M_x\cap K_{n}\setminus A$ and let $L_n$ be a nowhere dense subcontinuum of $X$ containing $x_n$ and $x_{n+1}$.  By hereditary unicoherence of $X$, $L_n\cap K_n$ is connected.  It follows that  $L=\bigcup _{n=0}^\infty L_n\cap K_n$ is connected and $\overline L$ is a continuum.  Each compact neighborhood in $K\setminus A$  intersects only a finite number of the sets $L_n\cap K_n$. Together with the assumption that $A$ is nowhere dense, this implies that $\overline L$  is nowhere dense. Thus  $\overline L\subset M_x$.  Further, $\overline L\cap A\neq\varnothing$ because the sequence $(x_n)$ has an accumulation point in $A$.   Therefore $A\cap M_x\neq\varnothing$.\end{proof}

\section{Open  meager composants}

\begin{ut}\label{c}Let $X$ be a hereditarily unicoherent continuum. If $M_x$ contains a dense open subset of $X$, then $M_x=X$.\end{ut}

\begin{proof} Suppose that $O\subset X$ is a dense open set and $O\subset M_x$.  Let $K$ be a connected component of  $X\setminus O$. Then $K$ is a nowhere dense subcontinuum of $X$.  If $K$ is not ample, then by Proposition \ref{a} its meager composant must intersect $O$ and we have  $K\subset M_x$. In the other case that $K$ is ample, $K\cap M_x\neq\varnothing$ by Proposition \ref{b}. Again $K\subset M_x$. This shows that each  connected component of $X\setminus O$ is contained in $M_x$. Hence $M_x=X$.\end{proof}

\begin{ut}[{{No proper open meager composant}}]\label{d}Let $X$ be a  hereditarily unicoherent continuum. If $M_x$ is open  then $M_x=X$. \end{ut}

\begin{proof} Suppose that $M_x$ is open. Then $M_x$ is a dense open meager composant of the hereditarily unicoherent continuum $\overline{M_x}$. By  Theorem \ref{c},  $M_x=\overline {M_x}$. Thus $M_x$ is a clopen subset of $X$. Since $X$ is connected, this implies $M_x=X$.\end{proof}

\begin{uc}\label{e}No tree-like continuum has a proper open meager composant.\end{uc}

\begin{ue}Theorem 3 cannot be  improved along the lines of:   \begin{center}\textit{If $U\subset M_x$ is open  then  $\overline U\subset M_x$.}\end{center} To see that this statement is false, consider the tree-like continuum  $X\subset [0,2]\times [0,1]$  that is depicted in Figure 1. The set $I=X\cap [0,1]^2$ is an indecomposable continuum  with endpoints $\langle 0,0\rangle$ and $\langle 1,1\rangle$ which belong to different composants of $I$.  The meager composant of $\langle 0,0\rangle$  in $X$ contains the open set $U=X\cap (1,2)\times (0,1)$. But $\overline U$ contains  $[1,2]\times \{1\}$ which belongs to a different meager composant of $X$.\end{ue}

\begin{figure}
\centering
\includegraphics[scale=0.37]{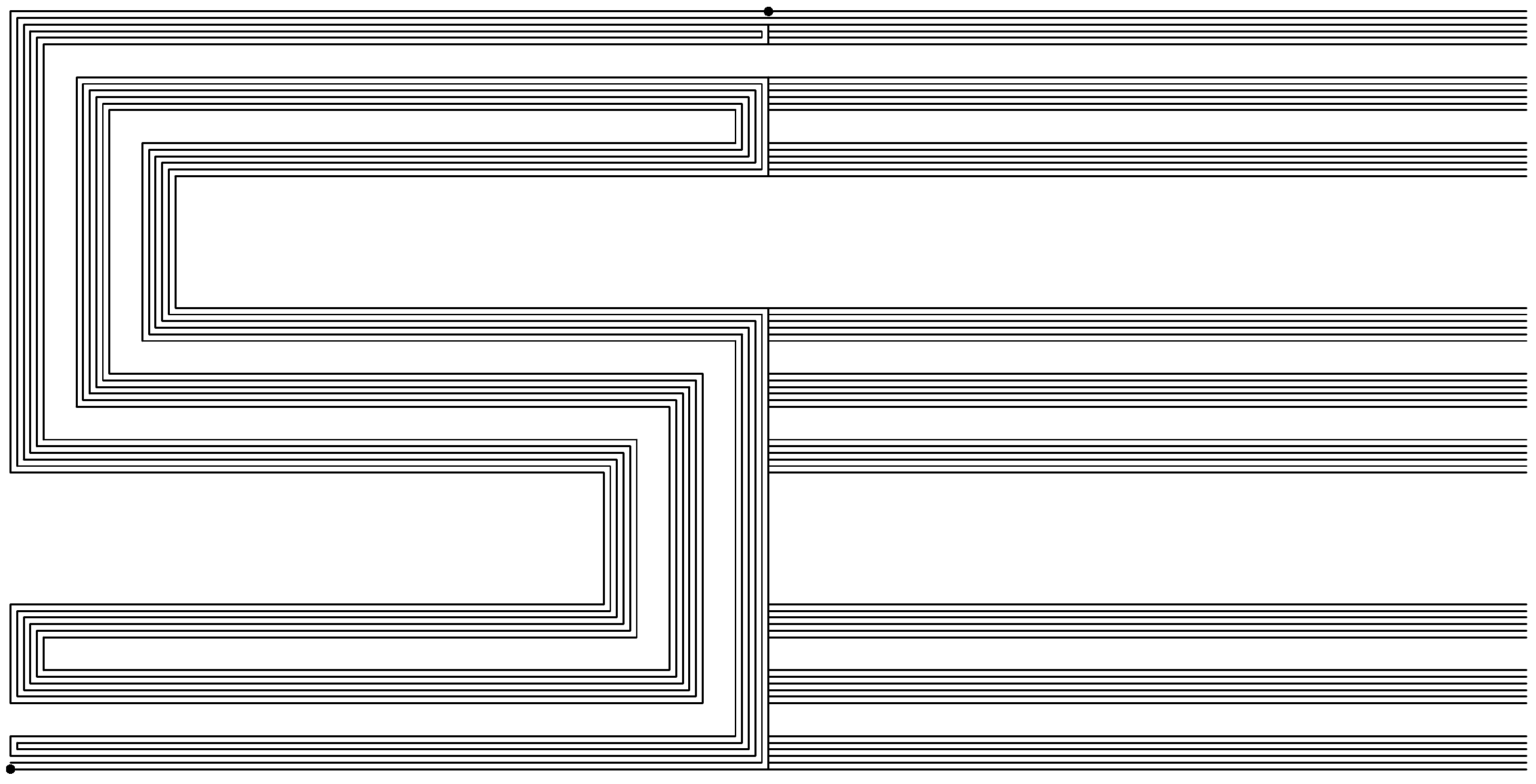}
\caption{The continuum $X$ for Example 1}
\end{figure}

\section{Non-closed meager composants}

In \cite[Section 4]{mou} it was shown that a hereditarily unicoherent continuum with a non-closed meager composant must contain an indecomposable continuum. 
 In this section we will prove the stronger statement that every such continuum has an  indecomposable subcontinuum with interior. The proof will involve the concept of  irreducibility.

 If $X$ is a continuum and $a,b\in X$ then $X$ is \textbf{irreducible between $a$ and $b$} if no proper subcontinuum of $X$ contains both $a$ and $b$.  In this case we define
\begin{align*}
A&=\big\{x\in X:X\text{ is irreducible between }x\text{ and } b\big\}\\
B&=\big\{x\in X:X\text{ is irreducible between }a\text{ and } x\big\}.
\end{align*}
By \cite[Theorem 11.4]{nad},  $A$ and $B$ are connected. Hence $\overline A$ and $\overline B$ are continua.

\begin{ul}\label{f}Let $X$ be a continuum that is irreducible between $a$ and $b$.  If  $\overline A$ has interior, then $\overline A$ is indecomposable (and likewise for $\overline B$).\end{ul}
 
 \begin{proof} Suppose that $\overline A$ has interior in $X$.  If $X$ is indecomposable then $\overline A=X$ and we are done. Assume now that  $X$ is decomposable. Then $C=X\setminus \overline A$ is connected and $b\in C$.  It follows that $\overline A$ is equal to the closure of its interior, otherwise $\overline C$ would be a proper subcontinuum of $X$ intersecting $A$ and containing $b$. We are now ready to show that $\overline A$ is indecomposable. For a contradiction suppose that  $\overline A=H\cup K$ where $H$ and $K$ are proper subcontinua of $\overline A$.  Without loss of generality  $K\cap \overline C\neq\varnothing$. Since $\overline A$ is the closure of its interior,  $H\setminus K$ has interior in $X$. Therefore $K\cup \overline C$ is a proper subcontinuum of $X$. It  intersects $A$ and contains $b$, which is a contradiction. Therefore $\overline A$ is indecomposable. \end{proof}

A standard consequence of Zorn's lemma is that for every two points $a$ and $b$ in a continuum $X$ there exists a subcontinuum of $X$ that is irreducible between $a$ and $b$ \cite[Exercise 4.35]{nad}. This fact will be used in the proof below.
 
\begin{ut}\label{g}Let $X$ be a hereditarily unicoherent continuum. If $M_x$ is not closed, then $\overline{M_x}$ contains an indecomposable continuum with interior in $X$. \end{ut}

\begin{proof} Suppose that $M_x$ is not closed. Let $K$ be a subcontinuum of $\overline{M_x}$ that is irreducible between $a\in M_x$ and  $b\in \overline{M_x}\setminus M_x$. Then $K$ has interior in $X$. And by Lemma \ref{f} may assume that $A$ and $B$ (the sets of irreducibility for $K$) are nowhere dense. Let $U$ be a non-empty open subset of $X$ that is contained in $K\setminus (\overline A\cup \overline B)$. Fix $a'\in U\cap M_x$, and let $L$ be a nowhere dense subcontinuum of $X$ with  $\{a,a'\}\subset L$.  Let $\delta=d(a',\overline B)$. For each $n\geq 1$ let $C_n$ be the connected component of $B$ in the set 
 $\{y\in K:d(a',y)\geq \delta/n\}$.  Note that $d(a',C_n)=\delta/n$ by the boundary bumping principle \cite[Theorem 5.4]{nad}.  Let $K'= \overline{\bigcup_{n=1}^\infty C_n}$. Then $K'$ is a continuum and $\{a',b\}\subset K'$.   By irreducibility of $K$ we have $L\cup K'=K$ and each $C_n$ misses $L$. It follows that $K'$ is irreducible between $a'$ and $b$. Let $$A'=\big\{y\in K':K'\text{ is irreducible between }y\text{ and } b\big\}.$$ We claim that the continuum $\overline{A'}$ has interior in $X$. This will be proved in two cases.

\begin{enumerate}[leftmargin=50pt]
\item[\underline{Case 1}:] $\overline{A'}$ is not ample in $K'$. Then by Proposition \ref{a} there is a nowhere dense continuum $M$ which intersects  $\overline{A'}$ and contains a point $x'\in K'\setminus \overline{A'}$. Let $P$ be a proper subcontinuum  of $K'$ that contains $x'$ and $b$. The open set $U\setminus P$ is non-empty as it contains $a'$, and by irreducibility of $K$  we have $K=L\cup \overline{A'} \cup M\cup P$.   Since $L$ and $M$ are nowhere dense, this means that $\overline{A'}$ has interior in $X$. 

\smallskip

\item[\underline{Case 2}:] $\overline{A'}$ is  ample in $K'$. Note that $C_n\cap M_x=\varnothing$ for each $n$ by irreducibility of $K$. So $C_n$ is nowhere dense in $X$ and thus $C_n\subset M_b$.  Therefore $M_b\cap K'$ is dense in $K'$. By Proposition \ref{b}, $\overline{A'}$ intersects $M_b$. Recall that  $\overline{A'}$ also contains $a'\in M_a$ and $M_a\neq M_b$. Therefore $\overline{A'}$ has interior in $X$. 
\end{enumerate}

By the preceding claim and Lemma \ref{f},  $\overline{A'}$ is an indecomposable continuum with interior in $X$. 
\end{proof}

We now establish the converse of Theorem \ref{g}.

\begin{ul}\label{h}Let $X$ be a hereditarily unicoherent continuum.  If $K$ is an indecomposable subcontinuum of $X$ with interior and $M_x\cap K\neq\varnothing$, then $M_x\cap K$ is a composant of $K$. In particular, $M_x$ is not closed.\end{ul}

\begin{proof} Suppose that $K$ is an indecomposable subcontinuum of $X$ with interior, and $y\in  M_x\cap K$. Let $\kappa$ be the composant  of $y$ in $K$. Then $\kappa$ is a meager composant of $K$ and hence  $\kappa\subset M_x$. On the other hand, if $L$ is any subcontinuum of $X$ meeting two different composants of $K$, then since $L\cap K$ is connected we have $K\subset L$ and thus $L$ has interior. This shows that $M_x\cap K\subset \kappa$. Therefore $M_x\cap K=\kappa$. Since $\kappa$ is not closed in $K$, we conclude that $M_x$ is not closed. \end{proof}

\begin{ut}[{{Characterization}}]\label{i}If $X$ is a hereditarily unicoherent continuum, then every meager composant of $X$ is closed iff every indecomposable subcontinuum of $X$ is nowhere dense. \end{ut}

\begin{proof} $(\Leftarrow)$: Suppose that $X$ has a non-closed meager composant. Then by Theorem \ref{g}, $X$ has an indecomposable subcontinuum with interior. 

 $(\Rightarrow)$: Suppose that $X$ has an indecomposable subcontinuum $K$ with interior. Let $x\in K$. By Lemma \ref{h}, $M_x$ is not closed.\end{proof}

\begin{ue}Hereditary unicoherence is critical  to each implication in Theorem \ref{i}. The  continuum featured in \cite[Section 5]{mou} is hereditarily decomposable but its meager composants are not closed. And it is easy to construct a continuum which has an indecomposable subcontinuum with interior and only $1$ (closed) meager composant.\end{ue}

\section{Closed meager composants}

 \begin{up}\label{j}Every closed meager composant of a continuum is ample.
 \end{up}
 
 \begin{proof} If $M_x$ is closed then apply Proposition \ref{a} to the continuum $M_x\subset X$.\end{proof}

\begin{ul}\label{k}Let $X$ be a hereditarily unicoherent continuum. If every meager composant of $X$ is closed, then the equivalence relation $E=\big\{\langle x,y\rangle\in X\times X: y\in M_x\big\}$ is closed in $X\times X$.\end{ul}

\begin{proof} Suppose that $\langle x_0,y_0\rangle, \langle x_1,y_1\rangle,\ldots  \in E$ and $\langle x_n,y_n\rangle\to \langle x,y\rangle\in X\times X$. And suppose that $M_x$ and $M_y$ are closed. We want to show $\langle x,y\rangle\in E$. To that end, for each $n\geq 0$ let $L_n$ be a nowhere dense subcontinuum of $X$ containing $x_n$ and $y_n$.  We may assume that there exists $N$ such that $L_n\cap M_x=\varnothing$ for all $n\geq N$.

For a contradiction, suppose that $\langle x,y\rangle\notin E$.   Then  $M_x$ and $M_y$ are disjoint and ample (Proposition \ref{j}). Thus there are continuum neighborhoods $C_1\supset C_2\supset \ldots$ and $D$  of $M_x$ and $M_y$ respectively such that $C_1\cap D=\varnothing$ and $\bigcap_{i=1}^\infty C_i =M_x$.   For each $i\geq 1$ choose $n_i\geq N$ such that $x_{n_i}\in C_{i}$ and $y_{n_i}\in D$. The set $$(L_{n_i}\cup L_{n_{i+1}}\cup D)\cap C_{i}=(L_{n_i}\cup L_{n_{i+1}})\cap C_{i}$$ contains $\{x_{n_i},x_{n_{i+1}}\}$ and  is connected by hereditary unicoherence of $X$. So $$L=\bigcup_{i=1}^\infty (L_{n_i}\cup L_{n_{i+1}})\cap C_{i}$$ is connected and $\overline L$ is a continuum.   Note that $\overline L$ is nowhere dense because  it misses the interior of $M_x$ and every compact neighborhood in $X\setminus M_x$ intersects only finitely many  constituents of $L$.  Thus $\{x_{n_1},x\}\subset \overline L$ implies $x_{n_1}\in M_x$. But   $L_{n_1}\cap M_x\neq\varnothing$ contradicts our earlier assumption.  Therefore $\langle x,y\rangle\in E$.\end{proof}

\begin{ut}[{{Meager decomposition}}] \label{l} Let $X$ be a hereditarily unicoherent continuum.  If  every indecomposable subcontinuum of $X$ is nowhere dense, then $$\mathcal M_X=\big\{M_x:x\in X\big\}$$ is an upper semi-continuous decomposition of $X$ and  the space $\mathcal M_X$ is a dendrite.\end{ut}

\begin{proof} Suppose that every indecomposable subcontinuum of $X$ is nowhere dense. By  Theorem \ref{i} every meager composant of $X$ is closed. By Lemma \ref{k} and  \cite[Exercise 7.17]{nad},   $\mathcal M_X$ is upper semi-continuous. Further, since each meager composant of $X$ is ample (Proposition \ref{j}), if  $\mathcal U$ is any open subset of $\mathcal M_X$ and $M_x\in \mathcal U$ then the connected component of $M_x$ in $\bigcup\mathcal U$ is an open union of meager composants. Therefore $\mathcal M_X$ is locally connected. It does not contain a simple closed curve because $X$ is hereditarily unicoherent and the quotient map $X\to\mathcal M_X$  is monotone. Therefore $\mathcal M_X$  is a dendrite.\end{proof}

\begin{uc}\label{m}If $X$ is a tree-like continuum and every indecomposable subcontinuum of $X$ is nowhere dense, then $X$ has an monotone upper semi-continuous decomposition into a dendrite whose elements are the  meager composants of $X$.\end{uc}

\section{Number of meager composants}

For a hereditarily unicoherent continuum $X$ the following conditions are easily seen to be equivalent:  $|\mathcal M_X|=1$; every two points of $X$ are contained in a nowhere dense subcontinuum of $X$;  every irreducible subcontinuum of $X$ is nowhere dense. The following theorem states that the only alternative to $|\mathcal M_X|=1$ is $|\mathcal M_X|=2^{\aleph_0}$.

\begin{ut}\label{n}If $X$ is a hereditarily unicoherent continuum then $|\mathcal M_X|\in \big\{1,2^{\aleph_0}\big\}$.\end{ut}

\begin{proof} The proof is by cases.

\begin{enumerate}[leftmargin=50pt]
\item[\underline{Case 1}:] Every indecomposable subcontinuum of $X$ is nowhere dense.  Then by Theorem \ref{l} $\mathcal M_X$ is a continuum. Every non-degenerate continuum has cardinality $2^{\aleph_0}$, so if $|\mathcal M_X|>1$ then $|\mathcal M_X|=2^{\aleph_0}$.

\item[\underline{Case 2}:] There is an indecomposable continuum $K\subset X$ with interior. Then  $K$ has $2^{\aleph_0}$ composants \cite{maz}, and no two composants of $K$  belong to the same meager composant of $X$ (Lemma \ref{h}). Therefore $|\mathcal M_X|=2^{\aleph_0}$.\qedhere
\end{enumerate}\end{proof}

\begin{uc}\label{o}Every tree-like continuum has exactly $1$ or $2^{\aleph_0}$ meager composants.\end{uc}

We end with an application to chainable continua. A continuum $X$ is \textbf{chainable} if for every $\varepsilon>0$ there are open sets $U_1,\ldots, U_n$ covering $X$  such that  $\text{diam}(U_i)<\varepsilon$ and $U_i\cap U_j\neq\varnothing$$\iff$$|i-j|\leq 1$ for all $i,j\leq n$. 

\begin{uc}\label{q}Every chainable continuum has $2^{\aleph_0}$ meager composants.\end{uc}
 
 \begin{proof}If $X$ is a chainable continuum, then $X$ is irreducible \cite[Theorem 12.4]{nad} and therefore $|\mathcal M_X|>1$. Additionally, $X$ is arc-like \cite[Theorem 12.11]{nad}. By Corollary \ref{o} we have $|\mathcal M_X|=2^{\aleph_0}$.
\end{proof}

\section*{Appendix: Homogeneous continua}

 A continuum $X$ is \textit{homogeneous} if for every two points $x,y\in X$ there is a homeomorphism $h$ of $X$ onto itself such that $h(x)=y$. Homogeneous continua form a class of spaces with very uniform structures, with fundamental examples such as the circle, Menger universal curve, pseudo-arc, circle of pseudo-arcs, and solenoids.    
  We establish the following classification in terms of meager composants.

\begin{ut}If $X$ is a homogeneous continuum, then precisely  one of the following holds.
\begin{enumerate}[label=\textnormal{(\alph*)}]
\item  $X$ has only one meager composant,
\item $X$ has proper dense meager composants, or 
\item  $X$ is a circle of indecomposable continua (i.e.\ there is continuous decomposition of $X$ into indecomposable subcontinua such that the decomposition space is a simple closed curve). 
\end{enumerate} \end{ut}

\begin{proof}  By \cite[Lemma 4.4]{pra3} the closures of meager composants $\overline{M_x}$ partition $X$ and are respected by homeomorphisms. From homogeneity of $X$ it follows that if $\overline{M_x}$ has interior then $\overline{M_x}$  is (cl)open and hence $\overline {M_x}=X$.  In this event $X$ falls into category (a) or (b).  Alternatively, if   $\overline{M_x}$ has empty interior then it is a nowhere dense subcontinuum of $X$, and so  $M_x=\overline{M_x}$.  By \cite[Theorem 1]{rog} the decomposition $\mathcal M_X$ is continuous, the space $\mathcal M_X$ is a  (non-degenerate) homogeneous continuum, and each $M_x$ is indecomposable. It remains to show that $\mathcal M_X$ is a circle. By Proposition \ref{j}, $M_x$ is ample and so $\mathcal M_X$ is locally connected. From monotonicity and  lower semi-continuity of  the decomposition  it can be seen that if $\mathcal L$ is a nowhere dense subcontinuum of $\mathcal M_X$ then $\bigcup \mathcal L$ is a nowhere dense subcontinuum of $X$. So the meager composants of $\mathcal M_X$ are singletons.  Since $\mathcal M_X$ is  path-connected, this implies that it contains an arc with interior. Therefore $\mathcal M_X$ is a circle \cite{ad} and the proof is complete.\end{proof}

A characterization similar to Theorem \ref{i} may exist for homogeneous continua. The following questions would need to be answered.

\begin{uq}Let $X$ be a homogeneous continuum. If $X$ has an indecomposable subcontinuum with interior, then is $M_x$  not closed?\end{uq}

\begin{uq}Let $X$ be a homogeneous continuum. If $X$ has proper dense meager composants, then does $X$ contain an indecomposable continuum? Is $X$ indecomposable?\end{uq}


\begin{thebibliography}{HD}

\bibitem{ad}R.D. Anderson, One-dimensional Continuous Curves and a Homogeneity Theorem, Ann. Math., 68 (1958), 1--16.

\bibitem{1}D. Bellamy, Questions in and out of context, Open Problems in Topology II, Elsevier (2007), 259--262.


\bibitem{eng}R. Engelking, General Topology, Revised and completed edition Sigma Series in Pure Mathematics 6, Heldermann  Verlag, Berlin, 1989.



\bibitem{sing}D.S. Lipham, Singularities of meager composants and filament composants, Topology Appl., Volume 260 (2019) 104--115.


\bibitem{maz} S. Mazurkiewicz, Sur les continus ind\'{e}composables, Fund. Math. 10 (1927), 305--310.

\bibitem{mou}C. Mouron and N. Ordo\~{n}ez, Meager composants of continua, Topology Appl., Volume 210 (2016) 292--310.

\bibitem{nad}S.B. Nadler Jr., Continuum Theory: An Introduction, Pure Appl. Math., vol. 158, Marcel Dekker, Inc., New York, 1992.

\bibitem{pra}J.R. Prajs and K. Whittington, Filament sets and homogeneous continua, Topology Appl. 154 (8) (2007) 1581--1591.


\bibitem{pra3}J.R. Prajs and K.  Whittington, Filament sets and decompositions of homogeneous continua. Topology Appl. 154(2007), no. 9, 1942--1956.

\bibitem{rog}J.T. Rogers Jr.,  Decompositions of homogeneous continua, Pacific J. Math. 99 (1982),137--144. 



\end{thebibliography}
\end{document}